%% file: paper.tex
\documentclass[11pt]{article}

\usepackage{amsmath,amssymb,amsfonts,textcomp,amsthm,xifthen,psfrag,graphicx,color}
\usepackage[T1]{fontenc}

\date{}

\input{header}

\title{Combining the DPG method with finite elements
\thanks{Supported by CONICYT through FONDECYT projects 1150056 and 1170672, Anillo ACT1118 (ANANUM),
       and BASAL project CMM, Universidad de Chile, Chile.}}
\author{
Thomas~F\"uhrer$^\dagger$
\and
Norbert Heuer\thanks{
Facultad de Matem\'aticas, Pontificia Universidad Cat\'olica de Chile,
Avenida Vicu\~na Mackenna 4860, Santiago, Chile,
email: {\tt \{tofuhrer,nheuer\}@mat.uc.cl}}
\and
Michael Karkulik\thanks{
Departamento de Matem\'atica, Universidad T\'ecnica Federico Santa Mar\'\i a,
Avenida Espa\~na 1680, Valpara\'\i so, Chile,
email: {\tt michael.karkulik@usm.cl}}
\and
Rodolfo Rodr\'\i guez\thanks{
Departamento de Ingenier\'\i a Matem\'atica and
Centro de Investigaci\'on en Ingenier\'\i a Matem\'atica,
Universidad de Concepci\'on, Casilla 160-C, Concepci\'on, Chile,
email: {\tt rodolfo@ing-mat.udec.cl}}}

\begin{document}
\maketitle
\begin{abstract}
We propose and analyze a discretization scheme that combines the discontinuous Petrov-Galerkin
and finite element methods. The underlying model problem is of general diffusion-advection-reaction
type on bounded domains, with decomposition into two sub-domains. We propose
a heterogeneous variational formulation that is of the ultra-weak (Petrov-Galerkin)
form with broken test space in one part, and of Bubnov-Galerkin form in the other.
A standard discretization with conforming approximation spaces and appropriate test spaces
(optimal test functions for the ultra-weak part and standard test functions for the Bubnov-Galerkin
part) gives rise to a coupled DPG-FEM scheme. We prove its well-posedness and quasi-optimal
convergence. Numerical results confirm expected convergence orders.


\end{abstract}

\section{Introduction}

The discontinuous Petrov-Galerkin method with optimal test functions (DPG method) is an approximation
scheme that makes the use of optimal test functions,
cf.~\cite{BarretM_84_ASP,CausinS_05_DPG,DemkowiczG_10_CDP}, feasible
by considering broken test norms \cite{DemkowiczG_11_CDP}. Optimal test functions are those
which maximize discrete inf-sup numbers, and the broken form of test spaces and norms allows
for their local calculation or approximation. In this form, the DPG method has been
developed by Demkowicz and Gopalakrishnan,
see the just cited references~\cite{DemkowiczG_10_CDP,DemkowiczG_11_CDP}.

The DPG method has been designed having in mind problems where standard methods suffer from
locking phenomena (of small inf-sup numbers) or, otherwise, require specific stabilization techniques.
This is particularly the case with singularly perturbed problems where DPG schemes have
made some contributions
\cite{DemkowiczH_13_RDM,ChanHBTD_14_RDM,BroersenS_14_RPG,BroersenS_15_PGD,HeuerK_RDM}.
Nevertheless, in the current form most of the schemes are not cheap to implement. On the
one hand, corresponding formulations have several unknowns as is the case with
mixed finite elements. On the other hand, the efficient approximation of optimal test functions
for singularly perturbed problems is ongoing research. For these reasons, advanced DPG techniques
are best used for specific problems whereas finite elements are hard to beat when solving
uniformly well-posed problems. Though, it has to be said, that in the latter cases
DPG schemes can also be efficient and are competitive in general,
cf. the software package developed by Roberts \cite{Roberts_14_CSF}.

In this paper we develop a discretization method that combines DPG techniques with standard
finite elements. In this way, one can restrict the use of more expensive DPG approximations
to regions where they are beneficial. Examples are, e.g., reaction-advection-diffusion problems
with small diffusivity in a reduced area, or transmission problems that couple a singularly
perturbed problem with an unperturbed problem. In a previous publication
\cite{FuehrerHK_CDB} we have proposed such a combination with boundary elements
to solve transmission problems of the Laplacian in the full space,
and studied a singularly perturbed case of reaction diffusion in \cite{FuehrerH_RCD}.
In this paper we follow the general framework from \cite{FuehrerHK_CDB}.
There, the basis is set by a heterogeneous variational formulation consisting of an ultra-weak
one in a bounded domain and variational boundary integral equations for the exterior
unbounded part. Here, we combine an ultra-weak formulation with a standard variational form.
We remark that this approach of combining different variational formulations has been
systematically analyzed in \cite{FuentesKDLT_CVF}. Indeed, it is not essential to use
an ultra-weak formulation for the DPG scheme, any well-posed formulation would work.
Though, the overall strategy in \cite{FuentesKDLT_CVF} is to employ
DPG techniques throughout whereas we combine different discretization techniques.

Having set our heterogeneous formulation, we proceed to rewrite it by using the so-called
trial-to-test operator (which maps the test space to the ansatz space). This is only done
for the ultra-weak formulation. The whole system then transforms into one where spaces
on the ansatz and test sides are identical. In this way, our heterogeneous variational
formulation fits the Lax-Milgram framework just as in \cite{FuehrerHK_CDB}.
We prove coercivity under the condition that the trial-to-test operator is weighted by
a sufficiently large constant. Then, quasi-optimal convergence of a discretized version
follows by standard arguments. When proving coercivity we follow steps that are similar
to the ones when studying the coupling of DPG with boundary elements.
But whereas \cite{FuehrerHK_CDB} analyzes only the Laplacian, here we set up the scheme
and prove coercivity for a general second-order equation of reaction-advection-diffusion type.
Throughout we assume that our problem is uniformly well posed, i.e., we do not study
variations for singularly perturbed cases as in \cite{FuehrerH_RCD}.
Also note that, since coefficients are variable, transmission problems can be treated the same way
by selecting the sub-domains accordingly. One only has to move the possibly non-homogeneous
jump data to the right-hand side functional.

The remainder of this paper is as follows. In Section~\ref{sec_math} we start by
formulating the model problem. A heterogeneous variational formulation is given in
\S\ref{sec_var}. There, we also state its well-posedness and coercivity (Theorem~\ref{thm_VF})
and briefly mention a simplified case where continuity across the sub-domain interface is
incorporated strongly (Corollary~\ref{cor_VF}).
The corresponding discrete DPG-FEM scheme is presented in
\S\ref{sec_DPG}. Its quasi-optimal convergence is announced in Theorem~\ref{thm_DPG}.
Most technical details and proofs are given in Section~\ref{sec_tec}.
In the last section we report on some numerical experiments.

Furthermore, throughout the paper, suprema are taken over sets excluding the null element,
and the notation $A\lesssim B$ is used to say that $A\leq C\cdot B$ with a constant $C>0$ which
does not depend on any quantity of interest. Correspondingly, the notation $A\gtrsim B$ is used.

\section{Mathematical setting and main results} \label{sec_math}
Let $\Omega\subset\R^d$, $d\in\{2,3\}$, be a bounded, simply connected polygonal/polyhedral
Lipschitz domain with boundary $\partial\Omega$, and
normal vector $\nn_\Omega$ on $\partial\Omega$ pointing outside of $\Omega$.
We consider the following elliptic problem of diffusion-advection-reaction type.
Given $f\in L_2(\Omega)$ find $u\in H^1_0(\Omega)$ such that
\begin{align} \label{prob}
  Au := \div\Bigl(-\talpha\nabla u + \bbeta u\Bigr) + \gamma u &= f
                   \quad \text{ in } \Omega.
\end{align}
Here, $L_2(\Omega)$ and $H^1_0(\Omega)$ denote standard Sobolev spaces, the latter with
zero trace on $\partial\Omega$.
Furthermore, all coefficients are supposed to be sufficiently regular, with
$\talpha(x)\in\R^{d\times d}$, $\bbeta(x)\in\R^d$, $\gamma(x)\in\R$ for $x\in\bar\Omega$.
We assume that all coefficients are uniformly bounded.
Furthermore, we assume that the symmetric part of $\talpha$ is positive definite and uniformly bounded
from below, with minimum eigenvalue larger than or equal to $\alpha_0>0$,
and that $\frac 12\div\bbeta+\gamma\ge 0$ in $\Omega$.
These conditions imply that the operator $A$ is bounded and coercive on $H^1_0(\Omega)$.

\subsection{Heterogeneous variational formulation}\label{sec_var}

In order to solve \eqref{prob} by a combination of DPG method and finite elements, we
formulate the problem in a heterogeneous way, using different variational forms in different
parts of the domain. For ease of illustration, we restrict ourselves to two Lipschitz sub-domains
$\Omega_1$, $\Omega_2$ (again of polygonal/polyhedral form, each with one connected component)
with boundaries $\partial\Omega_1$, $\partial\Omega_2$, as specified in Figure~\ref{fig_domain}.
There, also a notation for the boundary pieces is introduced. In particular, $\Gamma$ denotes
the interface between the sub-domains. The picture indicates that both sub-domains touch the
boundary of $\Omega$ (where the homogeneous Dirichlet condition is imposed), but this is not
essential. For instance, one sub-domain, $\Omega_2$, can be of annular type so that, in that case,
$\partial\Omega\subset\partial\Omega_2$ and $\Gamma=\partial\Omega_1$.
Other combinations can be analyzed without difficulty, also including Neumann conditions.
Nevertheless, since our analysis centers around proving coercivity of bilinear forms,
we need positivity of  the combined advection-reaction term on a sub-domain that does not touch
the Dirichlet boundary.

\noindent
{\bf Assumption 1.}
\emph{For $i=1,2$ there holds:\\
If $\mathrm{meas}(\Gamma_i)=0$ then there exists $\beta>0$ such that
    $\frac 12\div\bbeta+\gamma\ge \beta$ a.e. in $\Omega_i$.}
\medskip

\begin{figure}[htb] \label{fig_domain}
\begin{picture}(270,130)(0,0)
\thinlines
\put(90,0){\circle*{3}} \put(155,130){\circle*{3}}
\thicklines
\put(25,0){\framebox(220,130){}}
\thinlines
\put(90,0){\line(1,2){65}}
\put(65,80){$\Omega_1$} \put(55,65){(DPG)}
\put(180,45){$\Omega_2$} \put(170,30){(FEM)}
\put(140,80){$\Gamma$} \put(0,50){$\Gamma_1$} \put(255,80){$\Gamma_2$}
\end{picture}
\begin{minipage}[b]{0.35\textwidth}
\begin{align*}
   &\bar\Omega = \bar\Omega_1\cup\bar\Omega_2,\quad \Omega_1\cap\Omega_2 = \emptyset,\\
   &\bar\Omega_1\cap\bar\Omega_2=\Gamma,\\
   &\bar\Omega_1\cap\partial\Omega = \bar\Gamma_1,\quad \bar\Omega_2\cap\partial\Omega = \bar\Gamma_2.
\end{align*}
\end{minipage}
\caption{Decomposition of the domain $\Omega$ into sub-domains.}
\end{figure}
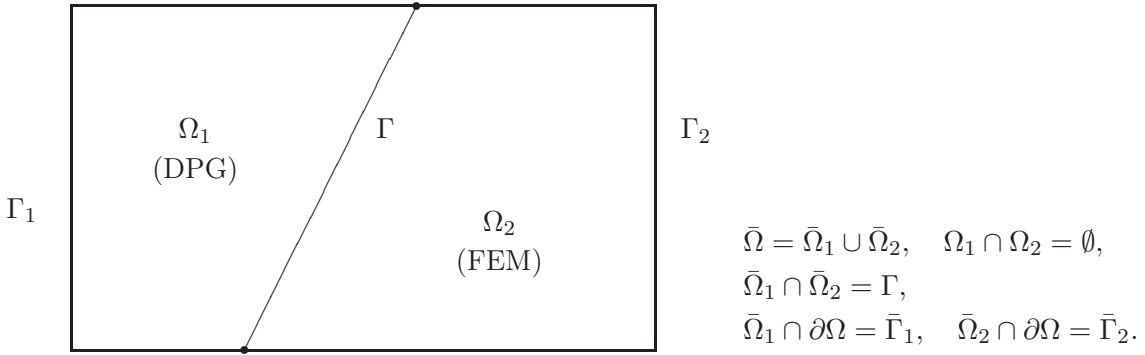

{\bf Standard and broken Sobolev spaces.}
Essential for the DPG method is to use broken test spaces. Therefore, at this early stage
we consider a partitioning $\TT_1$ of $\Omega_1$ into (regular non-intersecting) finite elements
$\el$ such that $\bar\Omega_1=\cup\{\bar\el;\; \el\in\TT_1\}$,
and with skeleton $\cS:=\{\partial\el;\;\el\in\TT_1\}$.

Before describing the variational formulation we introduce the Sobolev spaces we need.
For a domain $\omega\subset\Omega$ we use the standard spaces $L_2(\omega)$, $H^1(\omega)$,
$H^1_0(\omega)$, and $\HH(\mathrm{div},\omega)$.
The trace operator acting on $H^1(\omega)$ will be denoted simply by $(\cdot)|_{\partial\omega}$.
Then we define the trace space $H^{1/2}(\partial\omega) := H^1(\omega)|_{\partial\omega}$
and its dual space
$H^{-1/2}(\partial\omega) := \bigl(H^{1/2}(\partial\omega)\bigr)'$
with canonical norms.
The duality pairing on $\partial\omega$ is $\dual{\cdot}{\cdot}_{\partial\omega}$
and extends the $L_2(\partial\omega)$ bilinear form. Correspondingly,
$\vdual{\cdot}{\cdot}_\omega$ is the $L_2(\omega)$ bilinear form.

We also need $H^1_D(\Omega_i)$ consisting of $H^1$-functions with
vanishing trace on $\Gamma_i$ ($i=1,2$).
Vector-valued spaces and functions will be denoted by bold symbols.
Connected with $\TT_1$ we use the product spaces $H^1(\TT_1)$ and $\HH(\mathrm{div},\TT_1)$ 
with corresponding product norms.

Now, related with $\TT_1$ are the skeleton trace spaces
\begin{align*}
  H^{1/2}(\cS) &:=
  \Big\{ \wat u \in \Pi_{\el\in\TT_1}H^{1/2}(\partial\el);\;
         \exists w\in H^1(\Omega) \text{ such that } 
         \wat u|_{\partial\el} = w|_{\partial\el}\; \forall \el\in\TT_1 \Big\},\\
  H^{-1/2}(\cS) &:=
  \Big\{ \wat\sigma \in \Pi_{\el\in\TT_1}H^{-1/2}(\partial\el);\;
         \exists \qq\in\HH(\mathrm{div},\Omega) \text{ such that } 
         \wat\sigma|_{\partial\el} = (\qq\cdot\nn_{\el})|_{\partial\el}\; \forall\el\in\TT_1 \Big\}
\end{align*}
and
\begin{align*}
  H^{1/2}_{00}(\cS) &:=
  \bigl\{ \wat u\in H^{1/2}(\cS);\; \wat u|_{\partial\Omega_1} = 0 \bigr\},\\
  H^{1/2}_D(\cS) &:=
  \bigl\{ \wat u\in H^{1/2}(\cS);\; \wat u|_{\Gamma_1} = 0 \bigr\}.
\end{align*}
Here, $\nn_\el$ is the exterior unit normal vector on $\partial\el$, and
$(\qq\cdot\nn_{\el})|_{\partial\el}$ indicates the standard way of defining normal traces
of $\HH(\mathrm{div},\el)$-functions.
The notation $\wat u|_{\partial\Omega_1} = 0$ (resp. $\wat u|_{\Gamma_1} = 0$)
is to be understood in the sense that $\wat u$ is a $\TT_1$-piecewise trace of an element of
$H^1_0(\Omega_1)$ (resp. of $H^1_D(\Omega_1)$).
These trace spaces are equipped with the norms
\begin{subequations} \label{Hpm}
\begin{align}
  \norm{\wat u}{H^{1/2}(\cS)} &:=
  \inf \left\{ \norm{w}{H^1(\Omega)};\; w\in H^1(\Omega) \text{ such that }
               \wat u|_{\partial\el}=w|_{\partial\el}\; \forall\el\in\TT_1 \right\},\\
  \norm{\wat\sigma}{H^{-1/2}(\cS)} &:=
  \inf \left\{ \norm{\qq}{\HH(\mathrm{div},\Omega)};\; \qq\in\HH(\mathrm{div},\Omega) \text{ such that }
               \wat\sigma|_{\partial\el}=(\qq\cdot\nn_{\el})|_{\partial\el}\; \forall\el\in\TT_1 \right\},
\end{align}
\end{subequations}
and analogously for $H^{1/2}_{00}(\cS)$ and $H^{1/2}_D(\cS)$.
For functions $\wat u\in H^{1/2}(\cS)$, $\wat\sigma\in H^{-1/2}(\cS)$ (they are elements of product spaces)
and $\ttau\in\HH(\mathrm{div},\TT_1)$, $v\in H^1(\TT_1)$ we use the duality pairings
\begin{align*}
  \dual{\wat u}{\ttau\cdot\nn}_\cS
  := \sum_{\el\in\TT_1}\dual{\wat u|_{\partial\el}}{\ttau\cdot\nn_\el}_{\partial\el},\quad
  \dual{\wat\sigma}{v}_\cS
  := \sum_{\el\in\TT_1}\dual{\wat\sigma|_{\partial\el}}{v}_{\partial\el}.
\end{align*}

{\bf Heterogeneous formulation in $\Omega_1\cup\Omega_2$.}
In $\Omega_1$, where the DPG method will be used, we consider an ultra-weak variational formulation.
As mentioned before, this is just for illustration as any other formulation of
primal, mixed, dual-mixed or strong type can be used and analyzed analogously to our case,
cf.~\cite[Section 2.3]{FuentesKDLT_CVF}.

The ultra-weak formulation requires additional independent unknowns
\begin{align} \label{add_var}
   \ssigma    :=\talpha\nabla u - \bbeta u\ \text{on}\ \Omega_1,\quad
   \wat u     := \Pi_{\el\in\TT_1} u|_{\partial\el},\quad
   \wat\sigma := \Pi_{\el\in\TT_1} (\ssigma\cdot\nn_\el)|_{\partial\el}.
\end{align}
Then we test the defining relation of $\ssigma$
with $\talpha^{-T}$ and $\ttau\in\HH(\mathrm{div},\TT_1)$, and equation \eqref{prob} with $v\in H^1(\TT_1)$.
Integrating by parts element-wise, and substituting
the corresponding terms by $\ssigma$, $\wat u$, and $\wat\sigma$, we obtain
\begin{align} \label{form1}
     \vdual{u}{\pwdiv\ttau + \bbeta\talpha^{-T}\ttau + \gamma v}_{\Omega_1}
   + \vdual{\ssigma}{\pwnabla v+\talpha^{-T}\ttau}_{\Omega_1}
   - \dual{\wat u}{\ttau\cdot\nn}_\cS - \dual{\wat\sigma}{v}_\cS
   = \vdual{f}{v}_{\Omega_1}.
\end{align}
Here, $\pwdiv$ and $\pwnabla$ denote the $\TT_1$-piecewise divergence and gradient operators,
respectively.

In $\Omega_2$ we use the standard primal formulation
\begin{align} \label{form2}
     \vdual{\talpha\nabla u - \bbeta u}{\nabla w}_{\Omega_2}
   + \vdual{\gamma u}{w}_{\Omega_2}
   - \dual{\nn_{\Omega_2}\cdot(\talpha\nabla u-\bbeta u)}{w}_{\partial\Omega_2}
   =
   \vdual{f}{w}_{\Omega_2}
\end{align}
for $w\in H^1_D(\Omega_2)$.

Solving \eqref{prob} in $\Omega$ is equivalent to solving (in appropriate spaces)
\eqref{form1} and \eqref{form2} with homogeneous Dirichlet condition on $\partial\Omega$
and transmission conditions on $\Gamma$.
These transmission conditions will be imposed in variational form. For the time being, we replace
$\nn_{\Omega_2}\cdot(\talpha\nabla u-\bbeta u)|_\Gamma$ by $-\wat\sigma|_\Gamma$ in \eqref{form2}.
Here, we slightly abuse the notation of $\wat\sigma$ noting that
$\dual{\wat\sigma}{v}_\cS = \dual{\wat\sigma}{v}_\Gamma$ for $v\in H^1(\Omega_1)$ with $v|_{\Gamma_1}=0$,
cf., e.g., \cite[Section 2.2]{FuehrerH_RCD}.

We formally distinguish between $u_1:=u|_{\Omega_1}$ and $u_2:=u|_{\Omega_2}$. Then, our
preliminary heterogeneous variational formulation consists in finding
\begin{align*}
   &(\uu,u_2)=(u_1,\ssigma,\wat u,\wat\sigma,u_2)\in U := U_1\times H^1_D(\Omega_2)\\
   &\text{with}\quad
   U_1 := L_2(\Omega_1)\times \LL_2(\Omega_1) \times H^{1/2}_D(\cS) \times H^{-1/2}(\cS)
\end{align*}
such that
\(
   \wat u|_\Gamma = u_2|_\Gamma
\)
and
\begin{align*}
     \vdual{u_1}{\pwdiv\ttau + \bbeta\talpha^{-T}\ttau + \gamma v}_{\Omega_1}
   + \vdual{\ssigma}{\pwnabla v+\talpha^{-T}\ttau}_{\Omega_1}
   - \dual{\wat u}{\ttau\cdot\nn}_\cS - \dual{\wat\sigma}{v}_\cS
   = \vdual{f}{v}_{\Omega_1},\\
     \vdual{\talpha\nabla u_2 - \bbeta u_2}{\nabla w}_{\Omega_2}
   + \vdual{\gamma u_2}{w}_{\Omega_2}
   + \dual{\wat\sigma}{w}_{\Gamma}
   =
   \vdual{f}{w}_{\Omega_2}\\
   \text{for any}\quad (\vv,w)\in V\times H^1_D(\Omega_2)
\end{align*}
with
\[
   \vv = (v,\ttau)\quad\text{and}\quad
   V := H^1(\TT_1)\times \HH(\mathrm{div},\TT_1).
\]
This formulation can be used to define the combined DPG-FEM discretization,
but requires that $\TT_1$ be compatible across $\Gamma$ with the finite element mesh
in $\Omega_2$. We therefore replace the continuity constraint $\wat u|_\Gamma = u_2|_\Gamma$
by a variational coupling on $\Gamma$ that is similar to a DG-bilinear form
involving jumps and fluxes across element boundaries. To this end we abbreviate
\begin{align}
   b(\uu,\vv)
   &:=
     \vdual{u_1}{\pwdiv\ttau + \bbeta\talpha^{-T}\ttau + \gamma v}_{\Omega_1}
   + \vdual{\ssigma}{\pwnabla v+\talpha^{-T}\ttau}_{\Omega_1}
   - \dual{\wat u}{\ttau\cdot\nn}_\cS - \dual{\wat\sigma}{v}_\cS,
   \nonumber\\
   c_2(u_2,w_2)
   &:=
     \vdual{\talpha\nabla u_2 - \bbeta u_2}{\nabla w_2}_{\Omega_2}
   + \vdual{\gamma u_2}{w_2}_{\Omega_2},
   \label{c2}\\
   L_1(\vv) &:= \vdual{f}{v}_{\Omega_1},\quad
   L_2(w_2) := \vdual{f}{w_2}_{\Omega_2},
   \nonumber
\end{align}
and define the coupling bilinear form
\begin{align} \label{d}
   &d(\uu,u_2;\ww,w_2)
   :=
   \dual{\wat\sigma}{w_2}_{\Gamma}
   +
   \dual{\wat\chi}{\wat u - u_2}_\Gamma
   +
   \frac 12\dual{\bbeta\cdot\nn_{\Omega_1} (\wat u-u_2)}{\wat w + w_2}_\Gamma
   \\ \nonumber
   &\text{for}\quad
   (\uu,u_2), (\ww,w_2)\in U\quad
   \text{with}\quad
   \uu=(u_1,\ssigma,\wat u,\wat\sigma),\
   \ww=(w_1,\cchi,\wat w,\wat\chi).
\end{align}
The final combined ultra-weak primal formulation of \eqref{prob} then reads
\begin{subequations} \label{VF_pre}
\begin{align}
   (\uu,u_2)=(u_1,\ssigma,\wat u,\wat\sigma,u_2)\in U:&\nonumber\\
   b(\uu,\vv) &= L_1(\vv)
                          &&\hspace{-6em}\forall\vv\in V,        \label{VF_pre1}\\
   c_2(u_2,w_2) + d(\uu,u_2;\ww,w_2) &= L_2(w_2)
                          &&\hspace{-6em}\forall (\ww,w_2)\in U. \label{VF_pre2}
\end{align}
\end{subequations}
We will also need the bilinear form for $\Omega_1$ that corresponds to $c_2(\cdot,\cdot)$:
\begin{align} \label{c1}
   c_1(u_1,w_1)
   &:=
     \vdual{\talpha\nabla u_1 - \bbeta u_1}{\nabla w_1}_{\Omega_1}
   + \vdual{\gamma u_1}{w_1}_{\Omega_1}
   \quad \bigl(u_1,w_1\in H^1(\Omega_1)\bigr).
\end{align}
For reference, we explicitly specify the strong form of \eqref{VF_pre1}:
\begin{align} \label{VF_pre1_strong}
   \uu:=(u_1,\ssigma,\wat u,\wat\sigma)\in U_1:\quad B\uu = L_1.
\end{align}
Following \cite{FuentesKDLT_CVF} one can show that \eqref{VF_pre} is equivalent to \eqref{prob}
so that, in particular, \eqref{VF_pre} has a unique solution.
However, since we will use different strategies for solving \eqref{VF_pre1} and \eqref{VF_pre2},
we need a slightly different representation.

To this end we define the trial-to-test operator $\Theta:\;U_1\to V$ by
\[
   \ip{\Theta\uu}{\vv}_V = b(\uu,\vv)\quad\forall \vv\in V.
\]
Here, $\ip{\cdot}{\cdot}_V$ denotes the canonical inner product in $V$.
Note that $\Theta=\cR^{-1}B$ with Riesz operator $\cR:\;V\to V'$.
Since $B$ is defined on $U_1$ without boundary condition along $\Gamma$ it has a non-trivial kernel,
and so does $\Theta$. Still, $\Theta:\;U_1\to V$ is surjective. Therefore, denoting by
$\Theta_\kappa:=\kappa\Theta$ the scaled trial-to-test operator (for $\kappa>0$ to be chosen),
an equivalent formulation is:
For given $\kappa>0$ find
$(\uu,u_2)\in U$ such that
\begin{align} \label{VF}
   &a(\uu,u_2;\ww,w_2) = L(\ww,w_2) \qquad\forall (\ww,w_2)\in U\\
   &\text{with}\quad
   a(\uu,u_2;\ww,w_2)
   :=
   b(\uu,\Theta_\kappa\ww) + c_2(u_2,w_2) + d(\uu,u_2;\ww,w_2)
   \nonumber\\
   &\text{and}\quad L(\ww,w_2) := L_1(\Theta_\kappa\ww) + L_2(w_2).
   \nonumber
\end{align}
One of our main results is the following theorem.

\begin{theorem} \label{thm_VF}
The variational formulation \eqref{VF} is well posed, and is equivalent to problem \eqref{prob}
in the following sense. If $u\in H^1_0(\Omega)$ solves \eqref{prob} then
$(\uu,u_2)=(u_1,\ssigma,\wat u,\wat\sigma,u_2)$, with $u_i:=u|_{\Omega_i}$ ($i=1,2$) and
$\ssigma,\wat u,\wat\sigma$ defined by \eqref{add_var}, satisfies $(\uu,u_2)\in U$ and solves \eqref{VF}.

Vice versa, if $(\uu,u_2)=(u_1,\ssigma,\wat u,\wat\sigma,u_2)\in U$ solves \eqref{VF}
then $u$ defined by $u|_{\Omega_i}:=u_i$ ($i=1,2$) satisfies $u\in H^1_0(\Omega)$ and solves \eqref{prob}.

Furthermore, for sufficiently large $\kappa>0$, the bilinear form
\(
   a(\cdot,\cdot)
\)
is $U$-coercive, i.e.,
\begin{align} \label{ell}
   a(\uu,u_2;\uu,u_2) \gtrsim \|(\uu,u_2)\|_U^2
   \quad\forall (\uu,u_2)\in U.
\end{align}
\end{theorem}

\begin{proof}
By the assumptions on $\Omega$, $f$, $\talpha$, $\bbeta$, and $\gamma$, problem \eqref{prob}
is uniquely solvable. Furthermore, by the derivation of \eqref{VF}, if $u\in H^1_0(\Omega)$
solves \eqref{prob} then $(\uu,u_2)$ as specified in the statement solves \eqref{VF}.
This can be seen by integrating by parts and noting that
$d(\uu,u_2;\ww,w_2)=\dual{\wat\sigma}{w_2}_\Gamma$ since $\wat u|_\Gamma=u_2|_\Gamma$, cf.~\eqref{d}.

The coercivity of $a(\cdot,\cdot)$ will be shown in Section~\ref{sec_pf_ell} under the assumption
that $\kappa>0$ is large enough. It is also straightforward to show that
this bilinear form is bounded on $U\times U$,
as is the linear functional $L$ on $U$. In that case the Lax-Milgram lemma proves the
well-posedness of \eqref{VF}.

Now, since $\kappa$ introduces only a scaling of the test functions $\Theta_\kappa\ww\in V$,
the variational formulation \eqref{VF} is actually independent of $\kappa\not=0$, and so is its well-posedness.
\end{proof}

As previously mentioned, the continuity constraint $\wat u|_\Gamma=u_2|_\Gamma$ can also be
imposed strongly. In this case the solution space is
\begin{align*}
   U^0 := \{(u_1,\ssigma,\wat u,\wat\sigma;u_2)\in U;\; \wat u|_\Gamma = u_2|_\Gamma\}
\end{align*}
and the coupling bilinear form reduces to
\begin{align*}
   d^0(\uu,w_2) := d(\uu,u_2;\ww,w_2) = \dual{\wat\sigma}{w_2}_{\Gamma}
   \quad\forall
   (\uu,u_2)=(u_1,\ssigma,\wat u,\wat\sigma,u_2),\ (\ww,w_2)\in U^0.
\end{align*}
The variational formulation becomes: For given $\kappa>0$ find $(\uu,u_2)\in U^0$ such that
\begin{align}
   &a^0(\uu,u_2;\ww,w_2) = L(\ww,w_2) \qquad\forall (\ww,w_2)\in U^0
   \label{VF0}\\
   &\text{with}\quad
   a^0(\uu,u_2;\ww,w_2)
   :=
   b(\uu,\Theta_\kappa\ww) + c_2(u_2,w_2) + d^0(\uu;w_2)
   \label{a0}\\
   &\text{and}\quad L(\ww,w_2) := L_1(\Theta_\kappa\ww) + L_2(w_2).
   \nonumber
\end{align}
Analogously as Theorem~\ref{thm_VF} one obtains the well-posedness of \eqref{VF0} and
coercivity of $a^0(\cdot,\cdot)$.

\begin{cor} \label{cor_VF}
The variational formulation \eqref{VF0} is well posed, and is equivalent to problem \eqref{prob}
in the following sense. If $u\in H^1_0(\Omega)$ solves \eqref{prob} then
$(\uu,u_2)=(u_1,\ssigma,\wat u,\wat\sigma,u_2)$, with $u_i:=u|_{\Omega_i}$ ($i=1,2$) and
$\ssigma,\wat u,\wat\sigma$ defined by \eqref{add_var}, satisfies $(\uu,u_2)\in U^0$ and solves \eqref{VF0}.

Vice versa, if $(\uu,u_2)=(u_1,\ssigma,\wat u,\wat\sigma,u_2)\in U^0$ solves \eqref{VF0}
then $u$ defined by $u|_{\Omega_i}:=u_i$ ($i=1,2$) satisfies $u\in H^1_0(\Omega)$ and solves \eqref{prob}.

Furthermore, for sufficiently large $\kappa>0$, the bilinear form
\(
   a^0(\cdot,\cdot)
\)
is $U^0$-coercive, i.e.,
\begin{align*} 
   a(\uu,u_2;\uu,u_2) \gtrsim \|(\uu,u_2)\|_U^2
   \quad\forall (\uu,u_2)\in U^0.
\end{align*}
\end{cor}

\subsection{Combined DPG-FEM discretization}\label{sec_DPG}

The coupled DPG-FEM method consists in solving \eqref{VF} within finite-dimensional
subspaces $U_\hp\subset U$. The indices $h$ and $p$ indicate that this can be
piecewise polynomial, conforming spaces of certain polynomial degrees.
Specifically, the components of $U_\hp$ that belong to the unknowns
$u_1,\ssigma,\wat u,\wat\sigma$ will be piecewise polynomial with respect to the
mesh $\TT_1$ and its skeleton $\cS$. On the other hand, the component of $U_\hp$ that
approximates $u_2$ is piecewise polynomial with respect to a mesh $\TT_2$ in $\Omega_2$.
In the current form we do not need compatibility of the meshes $\TT_1$, $\TT_2$ along $\Gamma$.
The discrete scheme then reads:
For given $\kappa>0$ find $(\uu_\hp,u_{2,\hp})\in U_\hp$ such that
\begin{align} \label{DPG}
   a(\uu_\hp,u_{2,\hp};\ww,w_2) = L(\ww,w_2) \quad\forall (\ww,w_2)\in U_\hp.
\end{align}
Note that this formulation includes the use of optimal test functions for the discretization
in $\Omega_1$, cf.~\eqref{VF_pre1} and the corresponding terms in \eqref{VF} with
trial-to-test operator $\Theta_\kappa$. On the other hand, the part of the problem that belongs
to $\Omega_2$ is solved by standard finite elements, cf. the corresponding relation~\eqref{VF_pre2}.

Our second main result is the following theorem.

\begin{theorem} \label{thm_DPG}
If $\kappa>0$ is sufficiently large then the scheme \eqref{DPG} is uniquely solvable and
converges quasi-optimally, i.e.,
\[
   \|\uu-\uu_\hp\|_{U_1} + \|u_2-u_{2,\hp}\|_{H^1(\Omega_2)}
   \lesssim
   \inf\{\|\uu-\ww\|_{U_1} + \|u_2-w_2\|_{H^1(\Omega_2)};\; (\ww,w_2)\in U_\hp\}.
\]
\end{theorem}

\begin{proof}
The statement is a direct implication of the $U$-coercivity of $a(\cdot,\cdot)$
for large $\kappa$ by Theorem~\ref{thm_VF}, the Lax-Milgram lemma and Cea's estimate.
\end{proof}

\begin{remark}
We note that also the discrete scheme can be changed to impose strongly the continuity of the approximations
of $\wat u$ and $u_2$ across $\Gamma$. This only requires compatibility of the meshes $\TT_1$ and $\TT_2$
along the interface, conforming subspaces $U_\hp\subset U^0$, and replacing the bilinear form $a(\cdot;\cdot)$
in \eqref{DPG} by the bilinear form $a^0(\cdot;\cdot)$, cf.~\eqref{a0}.
The quasi-optimal error estimate from Theorem~\ref{thm_DPG} then holds analogously.
\end{remark}

\section{Technical details and proof of coercivity}\label{sec_tec}

We start with recalling the $H^1_0(\Omega)$-coercivity of the full differential operator $A$.
This transforms into the following properties of the bilinear forms
$c_2$, $c_1$, cf.~\eqref{c2},~\eqref{c1}.

\begin{lemma} \label{la_ell_c}
The bilinear forms $c_1(\cdot,\cdot)$ and $c_2(\cdot,\cdot)$ satisfy
\begin{align*}
   c_i(u_i,u_i) + \frac 12 \dual{\bbeta\cdot\nn_{\Omega_i} u_i}{u_i}_{\Gamma}
   \gtrsim
   \|u_i\|_{H^1(\Omega_i)}^2
\end{align*}
for all $u_i\in H^1_D(\Omega_i)$ ($i=1,2$).
\end{lemma}

\begin{proof}
Noting that
\[
   \vdual{\bbeta u}{\nabla u}_{\Omega_i}
   =
   -\frac 12 \vdual{(\div\bbeta) u}{u}_{\Omega_i}
   +\frac 12 \dual{\bbeta\cdot\nn_{\Omega_i} u}{u}_{\partial\Omega_i}
   \qquad \bigl(u\in H^1(\Omega_i),\ i=1,2\bigr),
\]
there holds for $u_i\in H^1_D(\Omega_i)$ ($i=1,2$)
\begin{align*}
   c_i(u_i,u_i)
   &=
   \vdual{\talpha\nabla u_i - \bbeta u_i}{\nabla u_i}_{\Omega_i} + \vdual{\gamma u_i}{u_i}_{\Omega_i}
   \\
   &=
   \vdual{\talpha\nabla u_i}{\nabla u_i}_{\Omega_i}
   + \vdual{(\frac 12\div\bbeta + \gamma) u_i}{u_i}_{\Omega_i}
   - \frac 12 \dual{\bbeta\cdot\nn_{\Omega_i} u_i}{u_i}_{\Gamma}.
\end{align*}
The coercivity property then follows with the positivity of the symmetric part of $\talpha$ and
by using either the Poincar\'e-Friedrichs inequality and $\frac 12\div\bbeta+\gamma\ge 0$ in $\Omega_i$
(if $\mathrm{meas}(\Gamma_i)\not=0$)
or Assumption~1, i.e., $\frac 12\div\bbeta+\gamma\ge\beta_i>0$ in $\Omega_i$ ($i=1,2$).
\end{proof}

We continue with some properties of the operator $B$, cf.~\eqref{VF_pre1_strong}, when restricted
to the space incorporating homogeneous Dirichlet boundary conditions on the whole of
$\partial\Omega_1$, that is,
\begin{align} \label{U10}
   B:\;U_{1,0}:=L_2(\Omega_1)\times\LL_2(\Omega_1)\times H^{1/2}_{00}(\cS)\times H^{-1/2}(\cS)\ \to\ V'.
\end{align}

\begin{lemma} \label{la_B}
The operator $B:\;U_{1,0}\to V'$ is an isomorphism with $\|B\|_{\cL(U_{1,0},V')}$ and
$\|B^{-1}\|_{\cL(V',U_{1,0})}$ bounded independently of the mesh $\TT_1$.
\end{lemma}

\begin{proof}
This is a particular case of the different variational formulations studied in
\cite[Example 3.7]{CarstensenDG_16_BSF}. More generally, in \cite{CarstensenDG_16_BSF},
Carstensen, Demkowicz and Gopalakrishnan proved that ``breaking'' a continuous variational formulation
of a well-posed problem (by introducing broken test spaces) and using canonical trace norms,
this does not alter the well-posedness of the formulation.
\end{proof}

Let us introduce the trace space
$H^{1/2}_{00}(\Gamma):=H^1_0(\Omega)|_\Gamma$
with canonical norm.
To simplify the presentation of some technical details we will need the following trace operator,
\begin{align*}
    \trace:\;U_1\to H^{1/2}_{00}(\Gamma),\quad
    \trace(u,\ssigma,\wat u,\wat\sigma):=\wat u|_\Gamma.
\end{align*}
The boundedness of this operator is immediate, and is analogous to the case of
the Laplacian on a single domain considered in \cite[Lemma 4]{FuehrerHK_CDB}.

\begin{lemma} \label{la_trace}
The operator $\trace$ is bounded with bound independent of $\TT_1$.
\end{lemma}

In the following we identify the kernel of $B$ when acting on the full space $U_1$.
Let us recall that $A$ is the operator of our problem \eqref{prob}.
For given $\varphi\in H^{1/2}_{00}(\Gamma)$ we define its $A$-harmonic extension
$\ext(\varphi):=(u_1,\ssigma,\wat u,\wat\sigma)\in U_1$ by
\begin{subequations} \label{ext}
\begin{align}
   &u_1\in H^1_D(\Omega_1):\;
   Au_1=0\ \text{in}\ \Omega_1,\quad
   u_1=\varphi\ \text{on}\ \Gamma,\\
   &\ssigma = \talpha\nabla u_1-\bbeta u_1,\quad
   \wat u = u_1\ \text{on}\ \cS,\quad
   \wat\sigma = \ssigma\cdot\nn_\el\ \text{on}\ \partial\el\;\forall \el\in\TT_1.
\end{align}
\end{subequations}

\begin{lemma} \label{la_ext_ker}
The operator $\ext:\;H^{1/2}_{00}(\Gamma)\to U_1$ is bounded with bound independent of $\TT_1$.
Moreover, $\ext$ is a right-inverse of $\trace$, and the image of $\ext$ is the kernel of $B$,
$\ker B=\ext H^{1/2}_{00}(\Gamma)$.
\end{lemma}

\begin{proof}
These statements can be proved analogously to the case of the Laplacian,
cf.~\cite[Lemmas 11 and 12]{FuehrerHK_CDB}.
\end{proof}

Of course, we also need continuity of the bilinear forms $b(\cdot,\cdot)$, $c_2(\cdot,\cdot)$ and
$d(\cdot,\cdot)$. This is straightforward to show and has already been used in
the initial part of the proof of Theorem~\ref{thm_VF}. We only give the statement.

\begin{lemma} \label{la_cont}
The bilinear forms $b:\;U_1\times V\to\R$, $c_2:\;H^1(\Omega_2)\times H^1(\Omega_2)\to\R$,
and $d:\; U\times U\to\R$ are all uniformly (in $\TT_1$) bounded.
\end{lemma}

\subsection{Proof of coercivity, statement \eqref{ell} in Theorem~\ref{thm_VF}}\label{sec_pf_ell}

We are now ready to prove the $U$-coercivity of the bilinear form $a(\cdot,\cdot)$,
cf.~\eqref{VF}. We adapt the procedure from \cite{FuehrerHK_CDB} to our situation.

Let $(\uu,u_2)=(u_1,\ssigma,\wat u,\wat\sigma,u_2)\in U$ be given.
We start with the simple estimate
\begin{align} \label{pf0}
   \|(\uu,u_2)\|_U
   &\le
   \|\uu\|_{U_1} + \|u_2\|_{H^1(\Omega_2)}
   \le
   \|\uu-\ext\trace(\uu)\|_{U_1} + \|\ext\trace(\uu)\|_{U_1} + \|u_2\|_{H^1(\Omega_2)}.
\end{align}
By Lemma~\ref{la_ext_ker}, the $\wat u$-component of $\uu-\ext\trace(\uu)$ has zero trace on
$\partial\Omega_1$, i.e.,
\(
   \uu - \ext\trace(\uu)\in U_{1,0},
\)
cf.~\eqref{U10}. Combining Lemmas~\ref{la_B} and \ref{la_ext_ker} this gives
\begin{align} \label{pf1}
   \|\uu-\ext\trace(\uu)\|_{U_1}
   \le
   \|\uu-\ext\trace(\uu)\|_{U_{1,0}}
   \lesssim
   \|B\uu\|_{V'}
   =
   b(\uu,\Theta\uu)^{1/2}.
\end{align}
The last identity is due to the well-known relations of the trial-to-test operator $\Theta$,
\[
   \|B\uu\|_{V'} = \sup_{\vv\in V} \frac {b(\uu,\vv)}{\|\vv\|_V}
                 = \frac {b(\uu,\Theta\uu)}{\|\Theta\uu\|_V}, \quad
   \|\Theta\uu\|_V = \|\cR^{-1}B\uu\|_V = \|B\uu\|_{V'}.
\]
According to Lemma~\ref{la_ext_ker}, operator $\ext$ is bounded,
\begin{align} \label{pf2}
   \|\ext\trace(\uu)\|_{U_1}
   \lesssim
   \|\wat u\|_{H^{1/2}_{00}(\Gamma)}.
\end{align}
A combination of \eqref{pf0}, \eqref{pf1}, and \eqref{pf2} then gives
\begin{align} \label{pf3}
   \|(\uu,u_2)\|_U^2
   \lesssim
   b(\uu,\Theta\uu) + \|\wat u\|_{H^{1/2}_{00}(\Gamma)}^2 + \|u_2\|_{H^1(\Omega_2)}^2.
\end{align}
We continue by considering $\uu^e:=(u_1^e,\ssigma^e,\wat u^e,\wat\sigma^e):=\ext\trace(\uu)=\ext\wat u|_\Gamma$.
In particular, there holds $\|\wat u\|_{H^{1/2}_{00}(\Gamma)}\lesssim \|u_1^e\|_{H^1(\Omega_1)}$.
Noting that, cf.~\eqref{d},
\begin{align} \label{pf3b}
   d(\uu^e,u_2;\uu^e,u_2)
   &=
   \dual{\wat\sigma^e}{u_2}_{\Gamma}
   +
   \dual{\wat\sigma^e}{\wat u - u_2}_\Gamma
   +
   \frac 12\dual{\bbeta\cdot\nn_{\Omega_1} (\wat u-u_2)}{\wat u + u_2}_\Gamma
   \nonumber\\
   &=
   \dual{\wat\sigma^e}{\wat u}_\Gamma
   +
   \frac 12\dual{\bbeta\cdot\nn_{\Omega_1} \wat u}{\wat u}_\Gamma
   -
   \frac 12\dual{\bbeta\cdot\nn_{\Omega_1} u_2}{u_2}_\Gamma,
\end{align}
an application of Lemma~\ref{la_ell_c} gives
\begin{align*}
   \|\wat u\|_{H^{1/2}_{00}(\Gamma)}^2 + \|u_2\|_{H^1(\Omega_2)}^2
   &\lesssim
   c_1(u_1^e,u_1^e) + c_2(u_2,u_2)
   + \frac 12 \dual{\bbeta\cdot\nn_{\Omega_1} u_1^e}{u_1^e}_{\Gamma}
   + \frac 12 \dual{\bbeta\cdot\nn_{\Omega_2} u_2}{u_2}_{\Gamma}
   \nonumber\\
   &=
   c_1(u_1^e,u_1^e) + c_2(u_2,u_2) + d(\uu^e,u_2;\uu^e,u_2)
   - \dual{\wat\sigma^e}{\wat u}_\Gamma.
\end{align*}
Relation \eqref{pf3b} can also be written like
\begin{align*}
   d(\uu^e,u_2;\uu^e,u_2)
   &=
   d(\uu,u_2;\uu,u_2) + \dual{\wat\sigma^e-\wat\sigma}{\wat u}_\Gamma,
\end{align*}
so that the previous bound becomes
\begin{align*}
   \|\wat u\|_{H^{1/2}_{00}(\Gamma)}^2 + \|u_2\|_{H^1(\Omega_2)}^2
   &\lesssim
   c_1(u_1^e,u_1^e) + c_2(u_2,u_2) + d(\uu,u_2;\uu,u_2)
   - \dual{\wat\sigma}{\wat u}_\Gamma.
\end{align*}
Now, recalling the definitions of $c_1(\cdot,\cdot)$ (see \eqref{c1})
and the extension operator $\ext$ (see \eqref{ext}), integration by parts yields the expected relation
$c_1(u_1^e,u_1^e)=\dual{\wat\sigma^e}{\wat u}_\Gamma$. Therefore, continuing the estimate,
\begin{align} \label{pf4}
   \|\wat u\|_{H^{1/2}_{00}(\Gamma)}^2 + \|u_2\|_{H^1(\Omega_2)}^2
   &\lesssim
   c_2(u_2,u_2) + d(\uu,u_2;\uu,u_2) + \dual{\wat\sigma^e-\wat\sigma}{\wat u}_\Gamma.
\end{align}
The last term in \eqref{pf4} can be bounded by duality, the continuity of
$H^{-1/2}(\cS)\ni (\wat\sigma^e-\wat\sigma)\mapsto (\wat\sigma^e-\wat\sigma)|_\Gamma\in H^{-1/2}(\Gamma)$,
and relation \eqref{pf1}. This gives
\begin{align*}
   \dual{\wat\sigma^e-\wat\sigma}{\wat u}_\Gamma
   &\lesssim
   \|\uu-\ext\trace(\uu)\|_{U_1} \|\wat u\|_{H^{1/2}_{00}(\Gamma)}
   \lesssim
   b(\uu,\Theta\uu)^{1/2} \|(\uu,u_2)\|_U.
\end{align*}
Combining this bound with \eqref{pf3} and \eqref{pf4},
and applying Young's inequality, we find that
\begin{align*}
   \|(\uu,u_2)\|_U^2 
   &\lesssim
   \kappa b(\uu,\Theta\uu) + c_2(u_2,u_2) + d(\uu,u_2;\uu,u_2)\\
   &= b(\uu,\Theta_\kappa\uu) + c_2(u_2,u_2) + d(\uu,u_2;\uu,u_2)
\end{align*}
for a sufficiently large constant $\kappa>0$.
This proves the stated coercivity of $a(\cdot,\cdot)$.

\section{Numerical experiments}\label{sec_num}
In this section we report on two numerical experiments. In both of them we choose $d=2$ and,
starting from a manufactured solution, we compute the right-hand side function $f$. The solution of the
second experiment does not satisfy the homogeneous Dirichlet boundary condition. In this case,
we use a standard approach and extend the inhomogeneous Dirichlet datum into the domain
and then shift the resulting terms to the right-hand side. As discrete trial space we use
\begin{align*}
  U_\hp := P^0(\TT_1) \times \left[ P^0(\TT_1) \right]^2 \times S^1_D(\cS)\times P^0(\cS) \times S^1_D(\TT_2),
\end{align*}
where $\TT_1$ and $\cS$ are a mesh and its skeleton in $\Omega_1$ and $\TT_2$ is a mesh in $\Omega_2$.
Throughout, we use meshes $\TT_1$ and $\TT_2$ which are compatible on the interface
$\Gamma$ (although this is not necessary in our analysis).
In the definition of $U_\hp$, $P^k(\TT_1)$ denotes the space of $\TT_1$-piecewise polynomials of degree $k$,
$P^0(\cS)$ denotes the space of piecewise constant functions on $\cS$, and $S^1_D(\cS)\subset H^{1/2}_D(\cS)$
denotes the space of piecewise affine and continuous functions on $\cS$ which vanish on $\Gamma_1$.
The space $S^1_D(\TT_2)\subset H^1_D(\Omega_2)$ is the space of piecewise affine, globally continuous
functions on $\TT_2$ which vanish on $\Gamma_2$.
The trial-to-test operator $\Theta_\kappa = \kappa \cR^{-1}B$ with $\cR:V\rightarrow V'$ being the Riesz operator
is approximated using the discrete Riesz operator $\cR_\hp:V_\hp\rightarrow V_\hp'$ with a finite dimensional
space $V_\hp\subset V$, which we choose to be
\begin{align*}
  V_\hp := P^2(\TT_1)\times \left[ P^2(\TT_1) \right]^2.
\end{align*}
The resulting method is called \textit{practical DPG method}, and was analyzed in \cite{GopalakrishnanQ_14_APD}.
In the latter work, it was shown that the additional discretization error of using $V_\hp$ instead of $V$
does not degrade the convergence order.
Throughout, we use $\kappa=1$ and do not encounter difficulties with this choice.
Note that if $\left( u_1,\ssigma, \wat u,\wat\sigma,u_2 \right)$ denotes the exact solution of~\eqref{VF}
and $\left( u_{1,\hp},\ssigma_\hp,\wat u_\hp, \wat\sigma_\hp, u_{2,\hp} \right)\in U_\hp$ denotes the discrete
solution~\eqref{DPG}, then
by definition of the norm $H^{1/2}(\cS)$ it holds
\begin{align*}
  \| \wat u - \wat u_\hp \|_{H^{1/2}(\cS)} \leq \| u - I_{\TT_1} \wat u_\hp \|_{H^1(\Omega_1)}=:\textrm{err}(\wat u),
\end{align*}
where $I_{\TT_1} \wat u_\hp\in S^1_D(\TT_1)$ is the nodal interpolant of $\wat u_\hp$
with $(I_{\TT_1} \wat u_\hp)|_\cS=\wat u_\hp$.
Likewise,
\begin{align*}
  \| \wat\sigma - \wat\sigma_\hp \|_{H^{-1/2}(\cS)} \leq
  \| \ssigma - \mathbf{I}_{\TT_1} \wat\sigma_\hp \|_{\HH(\mathrm{div},\Omega_1)}=:\textrm{err}(\wat\sigma),
\end{align*}
where $\mathbf{I}_{\TT_1}\wat\sigma_\hp\in\cR\TT_0(\TT_1)$ is the lowest-order Raviart-Thomas interpolant
of $\wat\sigma_\hp$, i.e.,
$(\nn_\el\cdot\mathbf{I}_{\TT_1}\wat\sigma_\hp)|_{\partial\el}=\wat\sigma_\hp|_{\partial\el}$
for any $\el\in\TT_1$.
Furthermore, we plot the errors
\begin{align*}
  \textrm{err}(u_1) &:= \| u_1 - u_{1,\hp} \|_{L_2(\Omega_1)},\\
  \textrm{err}(\ssigma) &:= \| \ssigma - \ssigma_\hp \|_{\LL_2(\Omega_1)},\\
  \textrm{err}(u_2) &:= \| u_2 - u_{2,\hp} \|_{H^1(\Omega_2)},
\end{align*}
as well as the so-called \textit{energy error} of the DPG part
\begin{align*}
  \textrm{err}(\uu) := \sup_{\vv\in V} \frac{b(\uu-\uu_\hp,\vv)}{\|\vv\|_{V}} = \| \Theta_\kappa (\uu-\uu_\hp) \|_{V},
\end{align*}
cf.~\eqref{pf1}.
In both experiments, we use a sequence of meshes resulting from uniform mesh refinements.
The quasi-optimality result of Theorem~\ref{thm_DPG} and well-known approximation results then show that
\begin{align*}
  \| \uu-\uu_\hp \|_{U_1} + \| u_2 - u_{2,\hp} \|_{H^1(\Omega_2)} = \OO(h) = \OO(N^{-1/2}).
\end{align*}
Here, $N$ denotes the overall number of degrees of freedom of $U_\hp$.
Hence, $\textrm{err}(\cdot)=\OO(N^{-1/2})$ for all of the errors defined above.

\subsection{Experiment 1}

We choose $\Omega_1 := \left( 0,1 \right)\times\left( 0,1 \right)$,
$\Omega_2:=\left( 1,2 \right)\times\left( 0,1 \right)$ and use the exact solution
\begin{align*}
  u(x,y):=x(2-x)y(1-y).
\end{align*}
The remaining parameters in the equation~\eqref{prob} are chosen as
$\talpha=id$, $\bbeta=\left( xy,1 \right)^{\top}$, and $\gamma=1-\sin(\pi x)$.
In Figure~\ref{fig:1} we plot the errors versus the degrees of freedom on a double logarithmic scale.
As expected, all the errors behave like $\OO(N^{-1/2})$, which is plotted in black without markers.
In Figure~\ref{fig:1a}, we plot the error $\wat u-u_2$ on the coupling boundary $\Gamma$
for the case with mesh width $1/32$.

\begin{figure}[htb]
  \psfrag{errU1}{     $\textrm{err}(u_1)$}
  \psfrag{errSigma1}{     $\textrm{err}(\ssigma)$}
  \psfrag{errU2}{     $\textrm{err}(u_2)$}
  \psfrag{errUhat}{     $\textrm{err}(\wat u)$}
  \psfrag{errSigmaHat}{     $\textrm{err}(\wat\sigma)$}
  \psfrag{errDPG}{     $\textrm{err}(\uu)$}
  \psfrag{total degrees of freedom}{$N$}
  \centering
  \includegraphics[width=0.7\textwidth]{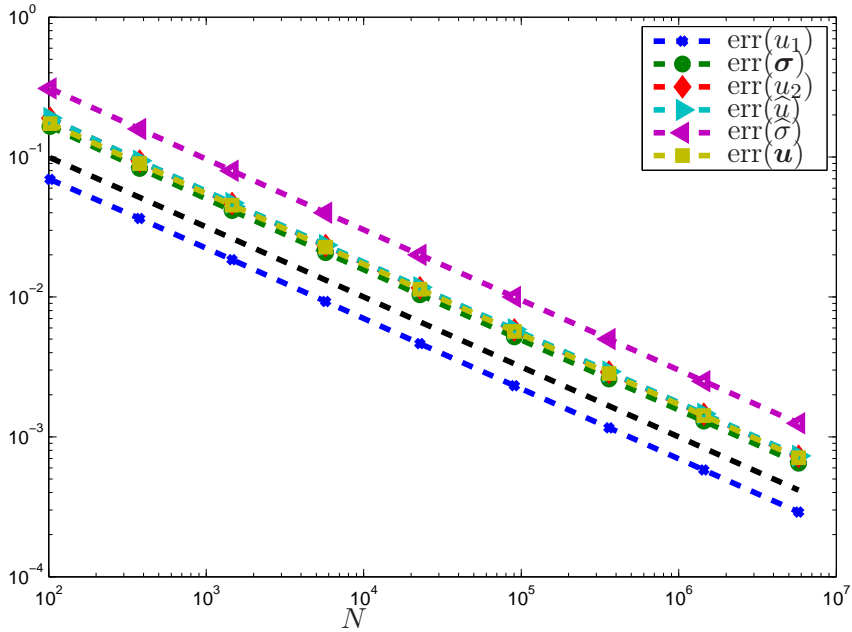}
  \caption{Error plots for Experiment~1. The black line without markers denotes
  $\OO(N^{-1/2})$, and $N$ is the total number of degrees of freedom.}
  \label{fig:1}
\end{figure}

\begin{figure}[htb]
  \centering
  \includegraphics[width=0.6\textwidth]{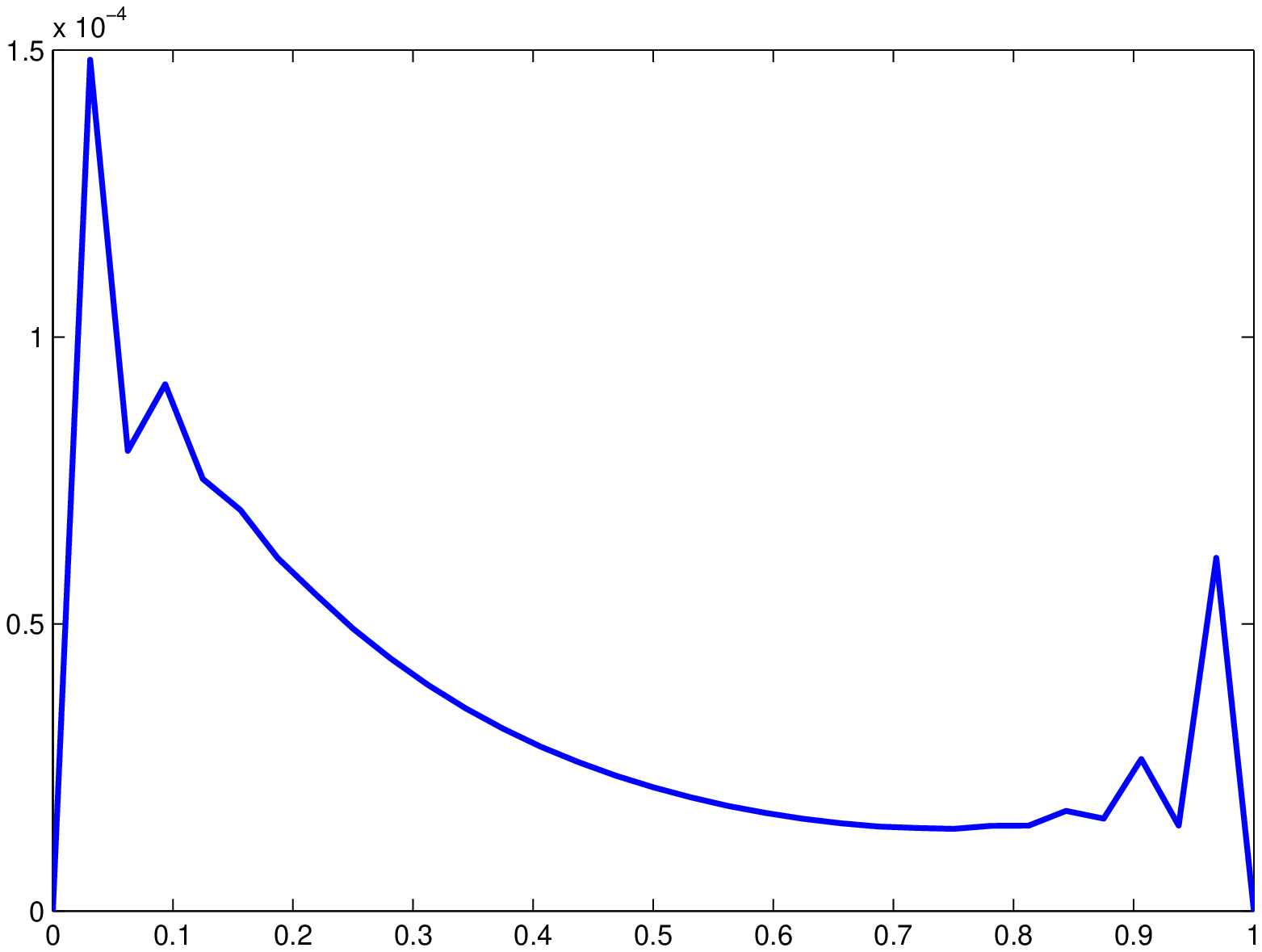}
  \caption{The jump $\wat u-u_2$ on the interface $\Gamma$ in Experiment~1.}
  \label{fig:1a}
\end{figure}

\subsection{Experiment 2}

We choose $\Omega_1 := (0.2, 0.7)\times(0.2,1.2)$ and $\Omega_2:=(0.7,1.2)\times(0.2,1.2)$
and use the exact solution
\begin{align*}
  u(x,y) := \arctan\left( \frac{1-\abs{(x,y)}}{\varepsilon} \right).
\end{align*}
The remaining parameters in the equation~\ref{prob} are chosen as
$\talpha = \varepsilon \cdot id$, $\bbeta = \exp(x)\left( \sin y, \cos y \right)$, $\gamma=0$,
and $\varepsilon=0.05$.
The exact solution $u$ has a curved layer of moderate width inside $\Omega$, see~Figure~\ref{fig_atan}.
In Figure~\ref{fig:2} we plot the errors versus the degrees of freedom on a double logarithmic scale.
Again, as expected, we obtain the convergence order $\OO(N^{-1/2})$.
In Figure~\ref{fig:2a}, we plot the error $\wat u-u_2$ on the coupling boundary $\Gamma$,
again for mesh width $1/32$.
Note that the layer of the exact solution cuts through $\Gamma$, and this is reflected in Fig.~\ref{fig:2a}.

\begin{figure}[htb]
  \centering
  \includegraphics[width=0.7\textwidth]{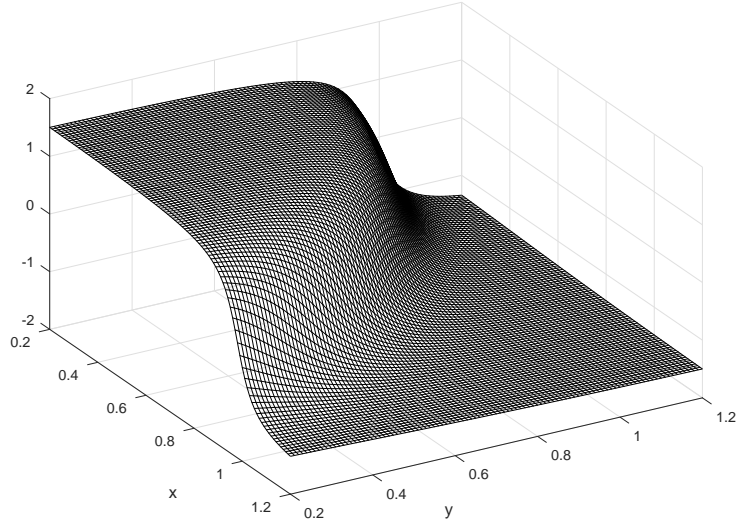}
  \caption{The exact solution from Experiment~2.}
  \label{fig_atan}
\end{figure}

\begin{figure}[htb]
  \psfrag{errU1}{     $\textrm{err}(u_1)$}
  \psfrag{errSigma1}{     $\textrm{err}(\ssigma)$}
  \psfrag{errU2}{     $\textrm{err}(u_2)$}
  \psfrag{errUhat}{     $\textrm{err}(\wat u)$}
  \psfrag{errSigmaHat}{     $\textrm{err}(\wat\sigma)$}
  \psfrag{errDPG}{     $\textrm{err}(\uu)$}
  \psfrag{total degrees of freedom}{$N$}
  \centering
  \includegraphics[width=0.6\textwidth]{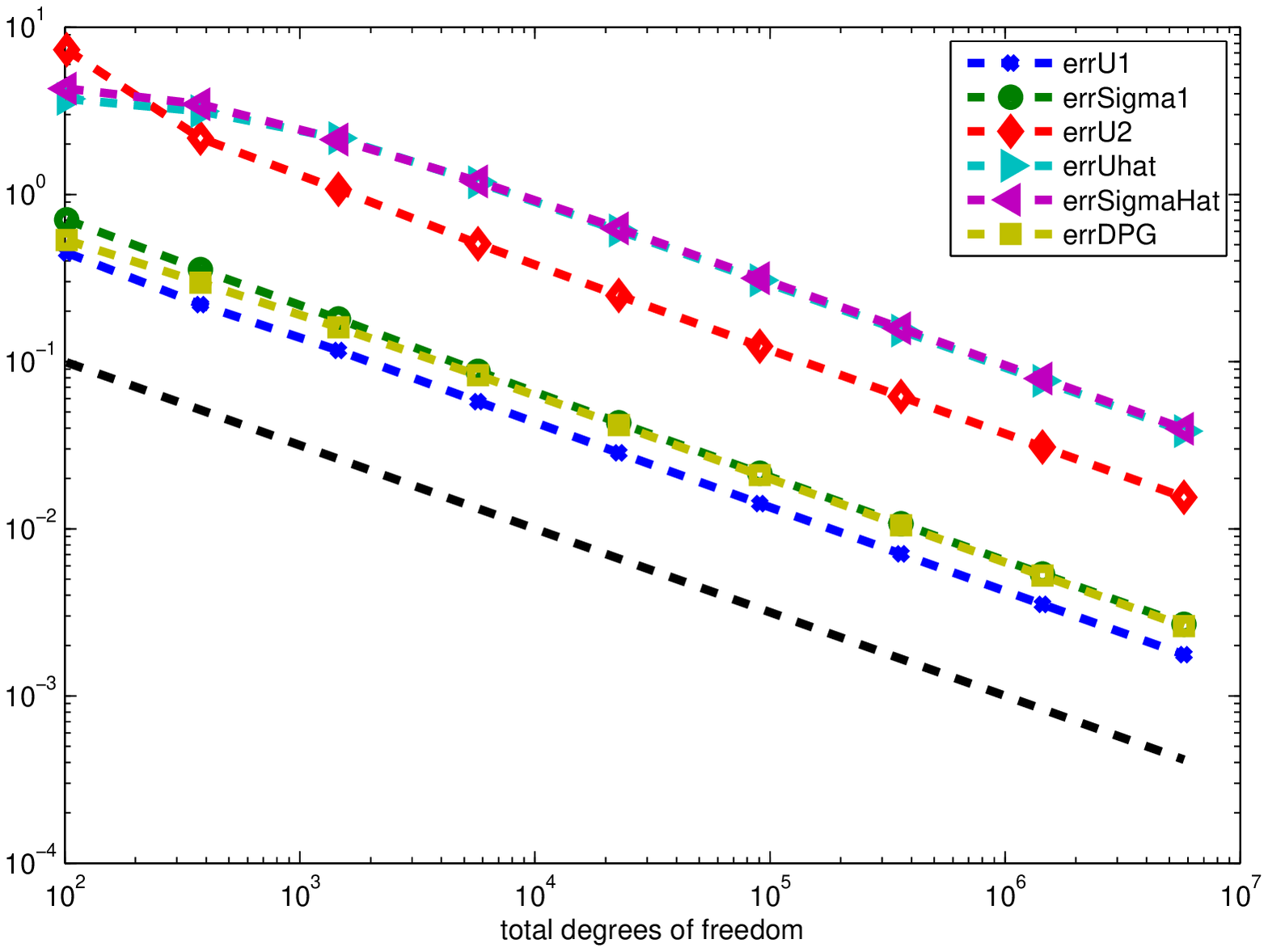}
  \caption{Error plots for Experiment~2. The black line without markers denotes
  $\OO(N^{-1/2})$, and $N$ is the total number of degrees of freedom.}
  \label{fig:2}
\end{figure}

\begin{figure}[htb]
  \centering
  \includegraphics[width=0.6\textwidth]{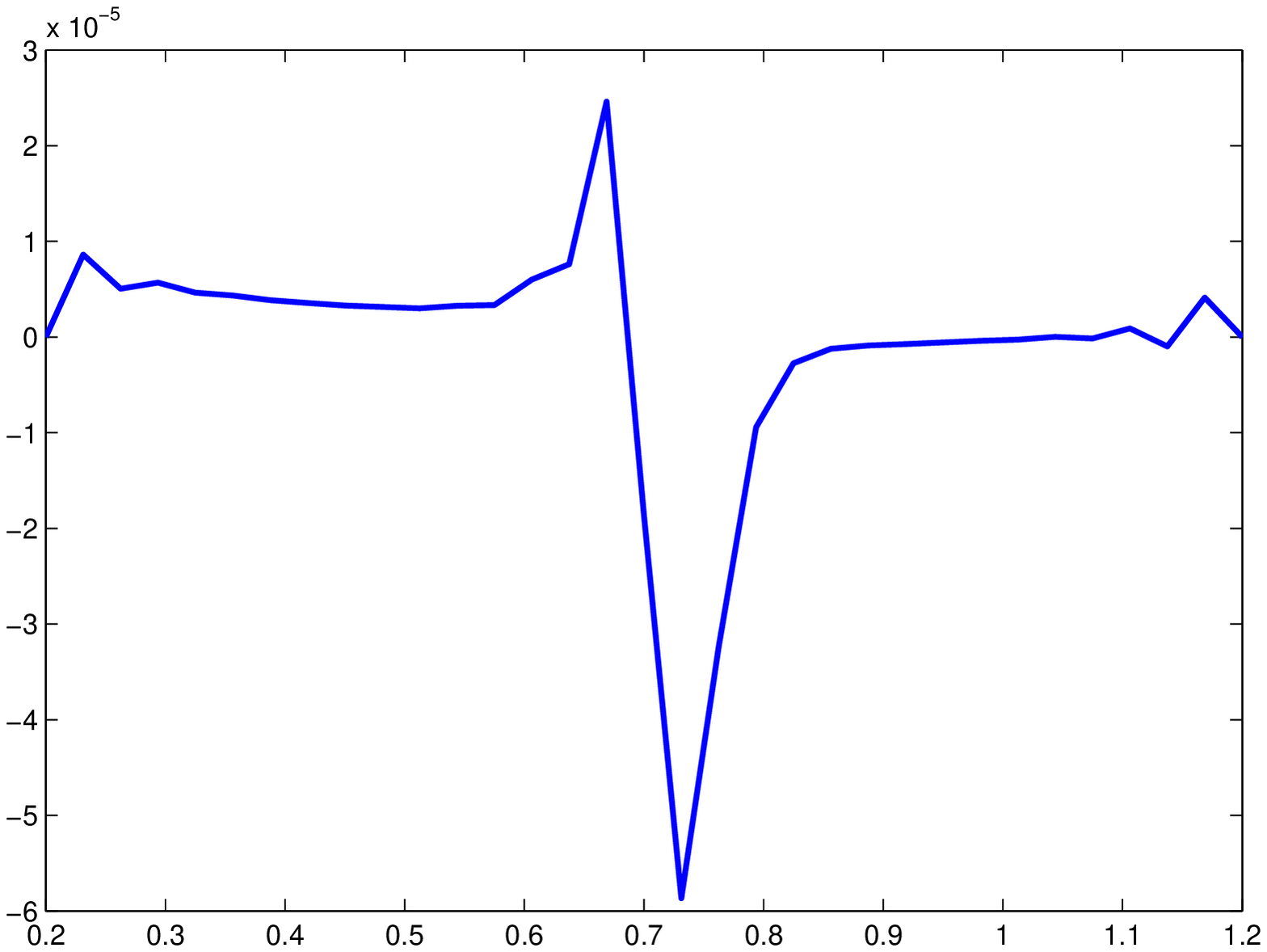}
  \caption{The jump $\wat u-u_2$ on the interface $\Gamma$ in Experiment~2.}
  \label{fig:2a}
\end{figure}

\bibliographystyle{abbrv}
\bibliography{/home/norbert/tex/bib/bib,/home/norbert/tex/bib/heuer}
\end{document}

%% file: header.tex
\newtheorem{theorem}{Theorem}
\newtheorem{lemma}[theorem]{Lemma}
\newtheorem{cor}[theorem]{Corollary}

\newtheorem{remark}[theorem]{Remark}
\theoremstyle{definition} 

\DeclareMathOperator{\ext}{\mathcal{E}}

\newcommand{\ip}[2]{\langle#1\hspace*{.5mm},#2\rangle}
\newcommand{\dual}[2]{\langle#1\hspace*{.5mm},#2\rangle}
\newcommand{\vdual}[2]{(#1\hspace*{.5mm},#2)}
\newcommand{\abs}[1]{\vert #1 \vert}
\newcommand{\norm}[3][]{#1\|#2#1\|_{#3}}

\newcommand{\wat}{\widehat}
\def\pwnabla{\nabla_\TT}

\def\div{{\rm div\,}}
\def\pwdiv{ {\rm div}_{\TT}\,}

\newcommand{\trace}{\mathrm{tr}_\Gamma}

\newcommand{\hp}{{hp}}

\newcommand{\R}{\ensuremath{\mathbb{R}}}

\newcommand{\HH}{\ensuremath{\mathbf{H}}}

\newcommand{\LL}{\ensuremath{\mathbf{L}}}
\newcommand{\cL}{\ensuremath{\mathcal{L}}}

\newcommand{\nn}{\ensuremath{\mathbf{n}}}
\newcommand{\vv}{\ensuremath{\boldsymbol{v}}}
\newcommand{\ww}{\ensuremath{\boldsymbol{w}}}
\newcommand{\TT}{\ensuremath{\mathcal{T}}}

\newcommand{\cS}{\ensuremath{\mathcal{S}}}
\newcommand{\cR}{\ensuremath{\mathcal{R}}}

\newcommand{\el}{\ensuremath{T}}

\newcommand{\OO}{\ensuremath{\mathcal{O}}}

\newcommand{\talpha}{{\underline{\alpha}}}

\newcommand{\bbeta}{{\boldsymbol\beta}}
\newcommand{\cchi}{{\boldsymbol\chi}}

\newcommand{\ssigma}{{\boldsymbol\sigma}}
\newcommand{\ttau}{{\boldsymbol\tau}}
\newcommand{\qq}{{\boldsymbol{q}}}
\newcommand{\uu}{\boldsymbol{u}}

\newcounter{constantsnumber}
\def\setc#1{
  \ifthenelse{\equal{#1}{poinc}}{C_{\rm edge}}{ 
   \refstepcounter{constantsnumber}
   \label{const#1}C_{\theconstantsnumber}}}

\def\c#1{
  \ifthenelse{\equal{#1}{poinc}}{C_{\rm edge}}{ 
    C_{\ref{const#1}}}}